\documentclass[12pt,a4paper]{amsart}

\usepackage{amsfonts, amsthm}
\usepackage{enumerate}
\usepackage[english]{babel}
\usepackage{amsmath}
\usepackage{amssymb}
\usepackage{array}
\usepackage[all]{xy}

\swapnumbers

\theoremstyle{theorem}
\newtheorem{theorem}{Theorem}[section]
\newtheorem{proposition}[theorem]{Proposition}
\newtheorem{lemma}[theorem]{Lemma}
\newtheorem{corollary}[theorem]{Corollary}

\theoremstyle{definition}
\newtheorem{definition}[theorem]{Definition}

\newcommand{\eee}{\ensuremath{\mathcal{E}}}

\newcommand{\mmm}{\ensuremath{\mathcal{M}}}

\newcommand{\ppp}{\ensuremath{\mathcal{P}}}

\newcommand{\ttt}{\ensuremath{\mathcal{T}}}
\newcommand{\uuu}{\ensuremath{\mathcal{U}}}
\newcommand{\vvv}{\ensuremath{\mathcal{V}}}

\newcommand{\zz}{\ensuremath{\mathfrak{Z}}}

\newcommand{\PM}{\ensuremath{\mathsf{ProbMet_{}}}}

\newcommand{\met}{\ensuremath{\mathsf{Met}}}

\newcommand{\N}{\ensuremath{\mathbb{N}}}

\renewcommand{\inf}{\myinf}

\DeclareMathOperator{\cl}{cl}

\DeclareMathOperator*{\myinf}{in\vphantom{p}f}

\renewcommand{\a}{\ensuremath{\alpha}}
\renewcommand{\b}{\ensuremath{\beta}}

\newcommand{\ve}{\ensuremath{\varepsilon}}

\renewcommand{\l}{\ensuremath{\lambda}}

\def\rw{\rightarrow}

\usepackage{pdfsync}
\begin{document}

\title{The category of Probabilistic metric spaces}
\author{E. Colebunders, R. Lowen}
\date{}
\maketitle

\vspace{.5cm}

\begin{center}
Abstract
\end{center}

\vspace{.1cm}

 \noindent \scriptsize{The paper is devoted to a categorical study of the category $\PM$ of probabilistic metric spaces.
The study is based on an isomorphic description of the category of probabilistic metric spaces. 
 The isomorphic description was obtained in \cite{CL2} and is in terms of objects that are sets endowed with a collection of distances, where
  the distances involved do not satisfy the triangle inequality but fulfil a mixed triangle condition instead. The morphisms are levelwise non-expansive maps.
 We show that the category of probabilistic metric spaces is a monotopological category over $\mathsf{Set}$. We describe the regular closure on a space $X$ in $\PM$ and prove that it coincides with the closure in the underlying strong topology.  This enables us to characterize the class $\mathcal{E}$ of all epimorphisms as the dense maps and the class $\mathcal{M}$ of all regular monomorphisms as the closed embeddings in terms of the closure operator. 
  We prove that the category $\mathsf{Met^{\infty}}$ of extended metric spaces with non-expansive maps is both coreflectively and reflectively embedded in $\PM.$

\vspace{.8cm}

 

\section{Introduction}
The paper is devoted to a categorical study of the category $\PM$ of probabilistic metric spaces and its subcategory $\mathsf{Met^{\infty}}$ of extended metric spaces.
The study is based on an isomorphic description of the category of probabilistic metric spaces in terms of objects that are sets endowed with a collection of distances labeled by the unit interval $]0,1]$ and with morphisms that are levelwise non-expansive. The distances involved do not satisfy the triangle inequality but fulfil a mixed triangle condition instead. This isomorphic description was given in \cite{CL2}, where it was derived from the study of probabilistic metrizability of uniform approach spaces.

 We show that $\PM$ is a monotopological category over $\mathsf{Set}$. We describe the regular closure on a space $X$ in $\PM$ and prove that it coincides with the closure in the strong topology underlying $X.$  This enables us to characterize the class $\mathcal{E}$ of all epimorphisms as the dense maps and the class $\mathcal{M}$ of all regular monomorphisms as the closed embeddings in terms of the closure operator.  It follows that $\PM$ is an $(\mathcal{E}, \mathcal{M})$-category, with $\eee$ the dense maps and $\mmm$ the closed embeddings that is cowell-powered. We prove that the category $\mathsf{Met^{\infty}}$ of extended metric spaces with non-expansive maps is both coreflectively and reflectively embedded in $\PM.$

 \newpage

\section{Preliminaries}

For background on concrete categories we refer to \cite{AHS}. Information on closure operators can be found in \cite{CGT}, \cite{DGT}, \cite{DG}, \cite{DT}. 
For more information on probabilistic metric spaces we refer to \cite{SS}, \cite{LSY}, \cite{N}, \cite{CS}. 

First we fix some notations.
A function $d: X \times X \rw [0, \infty]$  that is zero on the diagonal  is called a \emph{distance.} 
An \emph{extended pseudo metric} on a set $X$ is a map $d: X \times X \rw [0, \infty]$ which is zero on the diagonal, satisfies the triangle inequality and symmetry. If moreover $d$ satisfies separation, meaning it is only zero on the diagonal, it is called an extended metric and the category of extended metric spaces with non-expansive maps is denoted $\mathsf{Met}^\infty$.

We recall some terminology from \cite{SS} and \cite{LSY}. 

A function $\varphi: [0,\infty] \rw [0,1]$ is called a \emph{distance distribution} if $\varphi$ is monotone,  $\varphi(0) = 0, \varphi(\infty) =1$ and $\varphi$ is left-continuous on $]0,\infty[$.

A binary operation $*$ on the interval $[0,1]$ is a \emph{continuous t-norm} if $([0,1], *, 1)$ is a commutative monoid, which is continuous as a function on $[0,1]^2$ to $[0,1]$ with respect to the usual topologies and satisfies $p * q \leq p' * q'$ whenever $p \leq p'$ and $q \leq q'$ in $[0,1].$

A \emph{probabilistic metric space} is a set $X$ endowed with a map
\begin{equation}\label{PrMet}
\a: X \times X \times [0,\infty] \rw [0,1]
\end{equation}
and a continuous t-norm $*$ such that
\begin{enumerate}
\item [{(P1)}] $\a(x,y,-): [0,\infty] \rw [0,1] $ is a distance distribution,
\item [{(P2)}] $\a (x,x,- ) = \ve_0,$ the largest distance distribution,
\item [{(P3)}] $\a(x,y,r) = \a(y,x,r),$
\item [{(P4)}] $\a(x,y,-) = \ve_0 \Rightarrow  x = y,$
\item [{(P5)}] $\a(y,z,r) * \a(x,y,s) \leq \a(x,z,r+s),$
\end{enumerate}
for all $x,y,z \in X, r,s \in [0, \infty].$

A map $f:(X,\a) \rw (Y,\b)$ between probabilistic metric spaces is \emph{non-expansive} if 
\begin{equation}\label{nonex}
\a(x,x',t) \leq \b(f(x), f(x'),t),
\end{equation}
for all $x,x' \in X, t \in [0, \infty].$

\emph{In this paper we will always assume that the probabilistic metric spaces are defined with respect to some continuous t-norm  which we denote by $*$.}

The category of probabilistic metric spaces with non-expansive maps is denoted by $\PM.$

\section{An isomorphic description of the category of Probabilistic Metric Spaces}

Several attempts have been made to describe the category of probabilistic metric spaces with non-expansive maps isomorphically in terms of collections of extended metrics with levelwise non-expansive maps.  However in order to obtain an isomorphism based on extended metrics, one had to make restrictions regarding the axiom (P5), \cite{SS}, \cite{N}.

Here we stay with the axiom (P5) formulated for an arbitrary continuous t-norm as in \eqref{PrMet} and we work with collections of distances where the triangle condition of the members of the collection is replaced by a mixed triangle condition of the 
collection itself. 
We start from the isomorphic description of $\PM$ as in \cite{CL2} where the result was deduced from the characterization of probabilistic metrizability of (uniform) approach spaces. Since explicit proofs for collections of distances were omitted in \cite{CL2}, for clarity we will include explicit proofs here.

\begin{definition}\label{Z}
Consider the category which we temporarily denote $\mathfrak{Z}$ with objects, $(X, \{d_\l \ | \ 0< \l \leq 1\}),$ sets endowed with a collection of distances numbered by $ 0< \l \leq 1,$ where $\{d_\l \ | \ 0< \l \leq 1\}$
satisfies the following conditions: {\bf(US)}, {\bf(UD)}, {\bf(UT)}  for some continuous t-norm $*$  and {\bf(UH)}.\\

\item [{\bf(US)}] Symmetry: For every $0 < \l \leq 1$ and for $x,y \in X$, we have 
$$d_\l (x, y) = d_\l (y, x).$$
\item [{\bf(UD)}] Density:  For every $0 < \l \leq 1$ 
$$
d_\l = \inf _{\rho < \l} d_\rho.
$$ 
\item[{\bf(UT)}] t-Norm: For $0< \ve \leq 1,$ \ $0 < \l \leq 1$ and $0 < \l' \leq 1$ with 
$$
 (1 - \l') * ( 1 - \l) > 1 - \ve,
$$
 and for $x,y,z \in X$ we have
$$
d_\ve(x, z) \leq d_\l(x, y) + d_{\l'}(y, z).
$$
\item [{\bf(UH)}] Separation:
$$
x \not = y \ \text{then} \  \exists\l, \  0 < \l \leq 1, \ d_\l (x,y) > 0.
$$

Using the terminology of \cite{CL2} a morphism  $f : (X, \{d_\l \ | \ 0< \l \leq 1\}) \rw (Y, \{d'_\l \ | \ 0< \l \leq 1\})$ is called \emph{levelwise non-expansive}  if
\begin{equation}\label{morphisms}
\forall \l, \  d'_\l(f(x), f(x')) \leq d_\l (x,x'), \ \text{whenever} \ x, x' \in X.
\end{equation}
\end{definition}

\vspace{.7cm}

Let $(X,\a,*)$ be a probabilistic metric space. As in \cite{N}
\begin{equation}\label{distanceassociated}
(X, \{ d_\l \ | \ 0 < \l \leq 1 \})  
\end{equation}
 is associated with $(X,\a,*)$ by
$$  d_\l(x,y) = \inf \{ 0 \leq \gamma < \infty \ | \ \a(x,y,\gamma) > 1 - \l \}.$$

\begin{proposition}
We have the following equivalence, 
for $0 \leq \gamma < \infty$ and for $0 < \l \leq 1$ 
 \begin{equation}\label{isoequiv}
d_\l (x,y) < \gamma \Leftrightarrow \a(x,y,\gamma) > 1-\l.
\end{equation}
\begin{proof}
Suppose $d_\l (x,y)< \gamma,$ then there exists $\mu < \gamma$ satisfying $\a(x,y,\mu) > 1-\l.$
This implies $\a(x,y,\gamma) > 1-\l$ since $\a$ is nondecreasing.

To prove the other implication, assume $\a(x,y,\gamma) > 1-\l.$ Left continuity of $\a$ implies that there exists $\mu < \gamma$ with $ \a(x,y,\mu) > 1-\l.$ This implies\\
 $d_\l(x,y) \leq \mu < \gamma.$
\end{proof}
\end{proposition}

\begin{proposition}\label{delta} 
The transition in \eqref{distanceassociated} defines a concrete functor  $$\Delta: \PM \rw \frak{Z} $$ which maps an object $(X,\a, *)$ to  $(X, \{ d_\l \ | \ 0 < \l \leq 1 \})$ satisfying the conditions
{\bf(US)}, {\bf(UD)}, {\bf(UT)} and {\bf(UH)}.

\begin{proof}
{\bf{(US)}} Follows from (P3), since $\a(x,y,\gamma) = \a(y,x,\gamma)$ for all $x,y,\gamma.$\\
{\bf(UD)} Suppose $0 < \l' \leq \l \leq 1$ and $x \in X.$ First observe that in the infinite case, when $d_\l(x,y) = \infty$ then $\{ \gamma \ | \ \a(x,y,\gamma) > 1 - \l \} = \emptyset $ and then clearly also $\{ \gamma \ | \ \a(x,y,\gamma) > 1 - \l' \} = \emptyset $ and $d_{\l'}(x,y) = \infty.$
In the finite case  it is sufficient to apply \eqref{isoequiv} in order to conclude that $d_\l  \leq  d_{\l'}.$\\
From the previous observation we already have $d_\l \leq \inf _{\rho < \l} d_\rho.$
We show the reverse inequality. Let $x,y \in X$. In the infinite case with $d_\l(x,y) = \infty,$ there is nothing to show.
Next assume $d_\l(x,y) <  \gamma < \infty.$ Applying \eqref{isoequiv} we have $\a(x,y,\gamma) > 1 - \l.$ Then we can choose $\rho < \l$ with 
$\a(x,y,\gamma) > 1 - \rho.$ Clearly $d_\rho(x,y) < \gamma$ which implies $ \inf _{\rho < \l} d_\rho(x,y) < \gamma.$\\
{\bf(UT)} 
Let $0 < \ve \leq 1,$ \  $0 < \l \leq 1$ and $0 < \l' \leq 1$ satisfy
$$ (1 - \ve ) < (1 - \l') * ( 1 - \l).$$ 
For $x,y,z \in X$ arbitrary,
if the righthandside of the inequality $d_\ve(x,z) \leq d_\l(x,y) + d_{\l'}(y,z)$ is infinite, there is nothing to prove. Next assume $d_\l(x,y) < \gamma$ and  $d_{\l'}(y,z) < \gamma'$. Then we have $\a(y,z, \gamma') > 1 - \l'$ and $\a(x,y, \gamma) > 1 - \l.$
Now apply (P5)
$$
\a(x,z,\gamma + \gamma') \geq \a(y,z, \gamma')  * \a(x,y, \gamma) \geq (1 - \l') * ( 1 - \l)  > 1-\ve.
$$
Finally we get $d_\ve(x,z) < \gamma + \gamma'.$\\
{\bf(UH)} 
Suppose $x \not = y.$ That there exists $0 < \l \leq 1$ with $d_\l(x,y) > 0$ follows from the fact that by (P4) $\exists \gamma > 0$ with $\a(x,y,\gamma) \not = 1$ and hence $\exists \l$ with $\a(x,y,\gamma )\not  > 1 - \l.$\\

Next suppose $f: (X,\a) \rw (Y, \a')$ is a non-expansive map.
First observe that in case $d'_\l (f(x), f(x'))= \infty$ then for every $\gamma < \infty,$
$ \a(f(x), f(x'), \gamma) \not > 1 - \l.$
This implies that for every $\gamma < \infty, \ \a(x,x',\gamma) \not > 1 - \l.$ Hence we can conclude that $d_\l(x,x') = \infty.$

 In the finite case, applying \eqref{isoequiv} we obtain the following equivalences:
  \begin{eqnarray*}\label{morfequiv}
  f: (X,\a) \rw (Y, \a')\  \text{is non-expansive} \ &\Leftrightarrow & \forall t, \ \a(x,x',t) \leq \a'(f(x), f(x'), t )\\
&\Leftrightarrow& \forall t, \forall \l : \ \ \a(x,x', t) > 1 - \l  \Rightarrow\\
&& \a'(f(x), f(x') , t) > 1 - \l \\
&\Leftrightarrow& \forall t,  \forall \l :  \  d_\l(x,x') < t \Rightarrow  d'_\l(f(x), f(x')) < t\\
&\Leftrightarrow& \forall \l, \  d'_\l(f(x), f(x')) \leq d_\l (x,x').\\
\end{eqnarray*}
It follows that $\Delta (f) : (X, \{ d_\l \ | \ 0 < \l \leq 1 \}) \rw (Y, \{ d'_\l \ | \ 0 < \l \leq 1 \}) $ is levelwise non-expansive.
\end{proof}
\end{proposition}

Next we investigate the other direction.
Let $(X, \{ d_\l \ | \ 0 < \l \leq 1 \})$ be an object as defined in \ref{Z}. We associate a probabilistic metric space $(X, \b, *)$
 with it as follows:
 $\b (x,y, \infty ) = 1$ and further for $\gamma < \infty$:
\begin{equation}\label{isomorphbeta}
\b(x,y,\gamma) = \sup \{1 - \l \ | \ d_\l(x, y) < \gamma\} = 1 - \inf  \{\l \ | \ d_\l (x, y) < \gamma\}.
\end{equation}

\begin{proposition}
For $\gamma< \infty$ we have
\begin{equation} \label{morphi}
\b(x,y,\gamma) > 1 - \l   \Leftrightarrow  d_\l(x, y) < \gamma.
\end{equation}
\begin{proof} 
Let  $\sup \{1 - \rho \ | \ d_\rho(x,y) < \gamma\}  > 1 - \l.$ Then $\exists \rho < \l, \ d_\rho(x,y) < \gamma$ and applying {\bf(UD)} we have 
$d_\l(x,y) < \gamma.$

For the other implication,  assume $d_\l(x,y) < \gamma.$ Then again applying {\bf(UD)}\\
 $\exists \rho < \l, 
 \  \ d_\rho(x,y) < \gamma.$
Hence $1 - \l < 1 - \rho \leq \b(x,y,\gamma).$

\end{proof}
\end{proposition}

Before we show that the transition defines a concrete functor, we recall a lemma from \cite{CL2}.
\begin{lemma}\label{star}
Let $([0,1], *)$ be a continuous t-norm. For all $a,b,d \in ]0,1],$ the following are equivalent:
\begin{enumerate}
\item $d \geq a * b$
\item $\forall \l,\l' \in ]0,1], $\  if $a > 1-\l, \ b > 1-\l'$ then $d \geq (1-\l') * (1-\l)$
\item $\forall \rho  \in ]0,1],$ \ if $a * b > 1-\rho$ then $d \geq 1-\rho.$
\end{enumerate}
\end{lemma}

\begin{proposition}\label{Phi}
The transition in \eqref{isomorphbeta} defines a concrete functor $$\Phi: \mathfrak{Z} \rw \PM$$ which maps an object $(X, \{ d_\l \ | \ 0 < \l \leq 1 \})$ to $(X, \b, *)$. 
\begin{proof}
We first show that due to the properties of  $(X, \{ d_\l \ | \ 0 < \l \leq 1 \})$ listed above, $(X, \b, *)$  as in \eqref{isomorphbeta} defines a probabilistic metric space the sense of \eqref{PrMet}.

$\b(x,y,0) =0$ as $d_\l(x,y) \geq 0$ and $\b(x,y,\infty) =1$ by definition.

$\b(x,y,-)$ is nondecreasing as for $\gamma \leq \gamma',$  $\{ \l \ | \ d_\l(x,y) < \gamma\} \subseteq \{ \l \ | \ d_\l(x,y) < \gamma'\}.$

$\b(x,y,-)$ is left continuous on $]0, \infty[$. That $\sup_{s < \gamma} \b(x,y,s) \leq  \b(x,y,\gamma)$ is clear from the previous property. For the reverse inequality assume $\b(x,y,\gamma) > 1 - \l$. By \eqref{morphi} this implies that $d_\l(x,y) < \gamma$  and hence $\exists s < \gamma,  \ \ d_\l(x,y) < s.$ This implies that $\exists s <\gamma,  \  \b(x,y,s) > 1 - \l.$ So we have that $\sup_{s < \gamma} \b(x,y,s) > 1 - \l.$
From the previous claims we can conclude that (P1) holds.

To see that (P2) holds,  observe that for $\gamma >0$  we have\\
 $$\b(x,x,\gamma) = \sup \{1 - \l \ | \ \varphi_{\l,x} (x) < \gamma\}  = 1.$$

Clearly (P3) follows immediately from {\bf(US)}.

Next we check (P4). Suppose $x \not = y,$  from {\bf (UH)} it follows that $$\exists \l, \ d_\l(x,y) \not = 0.$$
Choosing $\gamma$ such that $0 < \gamma < d_\l(x,y), $ from \eqref{morphi} we obtain $\b(x,y,\gamma) \not = 1.$

Next we check that $\b$ satisfies (P5). Let $x,y,z \in X,$ and assume $\b(y,z,\gamma') > 1 - \l'$ and $\b(x,y,\gamma) > 1 - \l.$  
For every $ 0 < \ve \leq 1 $ satisfying 
$$
1-\ve  < (1 - \l') * ( 1 - \l),
$$
applying {\bf(UT)},  
we have
$$
d_\ve(x,z) \leq d_\l(x,y) + d_{\l'}(y,z) < \gamma' + \gamma.
$$
Hence
$$
\b(x,z, \gamma + \gamma') > 1 - \ve. 
$$
This implies
$$
\b(x,z, \gamma + \gamma') \geq  (1 - \l') * ( 1 - \l).
$$
Applying \ref{star}
we can conclude that $\b(x,z, \gamma + \gamma') \geq \b(y,z, \gamma')  * \b(x,y, \gamma). $\\

That the transition indeed defines a functor goes as follows.\\
Let $f: (X, \{ d_\l \ | \ 0 < \l \leq 1 \}) \rw (X, \{ d'_\l \ | \ 0 < \l \leq 1 \})$ be levelwise non-expansive.
First observe that in case $\gamma = \infty,$  $\b(x,x',\infty) =  \b'(f(x), f(x'), \infty )= 1$ by definition.
 In the finite case, applying \eqref{morphi} we obtain the following equivalences:
  \begin{eqnarray*}\label{phifunctor}
  f:  (X, \{ d_\l \ | \l \}) \rw (X, \{ d'_\l \ | \l \}) \  \text{levelw. n-exp}.&\Leftrightarrow & \forall \l, \  d'_\l(f(x), f(x')) \leq d_\l (x,x')\\
&\Leftrightarrow& \forall \gamma,  \forall \l : \  d_\l(x,x') < \gamma \Rightarrow\\
&&  d'_\l(f(x), f(x')) < \gamma\\
&\Leftrightarrow& \forall \gamma, \forall \l : \ \ \b(x,x', \gamma) > 1 - \l  \Rightarrow\\
&& \b'(f(x), f(x') , \gamma) > 1 - \l )\\
&\Leftrightarrow&\b(x,x', -) \leq \b'(f(x), f(x') , -)\\
&\Leftrightarrow&\Phi(f) : (X, \b ) \rw (Y, \b' )\  \text{n-exp.} 
\end{eqnarray*}
\end{proof}
\end{proposition}

\begin{theorem}
The category $\PM$ is isomorphic to the category $\mathfrak{Z}$.
\begin{proof}
We prove that $ \Phi \circ \Delta =\mathsf{id}_{\PM} $ and $ \Delta \circ \Phi = \mathsf{id}_{\zz}.$

Given a probabilistic metric space $(X, \a, *),$ applying $\Delta$ we obtain an object  $(X, \{d_\l \ | \ 0 <\l \leq 1 \}$ in $\zz$. Next applying $\Phi$ we obtain a probabilistic metric space $(X, \b, *),$ as described in \ref{delta} and \ref{Phi}. Applying \eqref{isoequiv} and \eqref{morphi}  for $\gamma < \infty$ we get 
$$
\b(x,y,\gamma) > 1 - \l  \Leftrightarrow d_\l (x,y) < \gamma \Leftrightarrow \a(x,y,\gamma) > 1-\l.
$$
Since $\b(x,y,\infty) = \a(x,y,\infty)= 1$, we have $\a = \b.$ 

For the other identity, we start with an object  $(X, \{d_\l \ | \ 0 <\l \leq 1 \})$ of $\zz$ and apply $\Phi$. Let $(X, \b, *)$ be the probabilistic metric space obtained. Applying $\Delta$ we get the object $(X, \{e_\l \ | \ 0 <\l \leq 1 \})$ from $\zz.$
Again applying \eqref{isoequiv} and \eqref{morphi}, for $\gamma < \infty$ we get
$$
e_\l(x,y) < \gamma \Leftrightarrow \b(x,y, \gamma) > 1 - \l \Leftrightarrow d_\l (x,y)  < \gamma,
$$
which implies $e_\l = d_\l$ for all $\l.$
\end{proof}

\end{theorem}

In view of the isomorphism, we use the terminology \emph{probabilistic metric space} for both the objects of $\PM$ and the objects of $\mathfrak{Z}$ and  from now onwards we call a morphism in $\PM$  \emph {non-expansive}.

\section{the category $\PM$.}

In this section we  show that the category $\PM$ of probabilistic metric spaces and non-expansive maps is Monotopological over $\mathsf{Set}.$ 
We recall from \cite{AHS} that a category, concrete over $\mathsf{Set}$ is \emph{monotopological }over $\mathsf{Set}$ provided that every structured point-separating source has a unique initial lift. 

\begin{theorem}\label{mono}
The category $\PM$ is monotopological and well fibred over $\mathsf{Set}.$
\begin{proof}
Let $(f_i: X \rw (X_i, (d^i_\l)_\l))_{i \in I}$ be an arbitrary pointseparating source with $(X_i, (d^i_\l)_\l))$ a probabilistic metric space for every $i \in I$.
On $X$ for $0< \xi \leq 1$ we put 
\begin{equation}\label{sup}
e_\xi (x,y)= \sup _{i \in I} d^i_\xi (f_i(x), f_i(y))
\end{equation}
and further for $0< \l \leq 1$
\begin{equation}\label{inf}
d_\l(x,y) = \inf_{\xi < \l} e_\xi(x,y) =  \inf_{\xi < \l} \sup _{i \in I} d^i_\xi (f_i(x), f_i(y)).
\end{equation}
We show that $(X,(d_\l)_\l)$ is a probabilistic metric space.
Clearly for every $\l$ $d_\l(x,x) = 0$ for $x \in X$ and $d_\l(x,y) = d_\l(y,x)$ for $x,y \in X.$

Clearly {\bf(UD)} holds by
$$
\inf_{\mu < \l} \inf_{\xi < \mu} \sup_{i \in I} d^i_\xi (f_i(x) , f_i(y) )= \inf_{\xi < \l} \sup_{i \in I} d^i_\xi (f_i(x) , f_i(y) ) = d_\l(x,y).
$$
In order to show that  for $(d_\l)_\l$  {\bf(UT)} is fulfilled, as a first step we prove {\bf(UT)} for $(e_\l)_\l.$
Let $\xi, \rho, \rho'$ be such that $(1-\xi )< (1 - \rho') * (1 - \rho)$ and $x,y,z \in X$ arbitrary.
For all $i \in I$ we have 
$$
d^i_\xi(f_i(x), f_i(z)) \leq d^i_\rho(f_i(x), f_i(y)) +d^i_{\rho'}(f_i(y), f_i(z)). 
$$
Hence for all $i \in I,$ 
$
d^i_\xi(f_i(x), f_i(z)) \leq e_\rho(x,y) +e_{\rho'} (y,z),
$
and finally $e_\xi (x,z) \leq e_\rho(x,y) +e_{\rho'} (y,z).$

Next we check {\bf(UT)} for $(d_\l)_\l.$ Let $\ve, \l, \l'$ be such that $(1-\ve )< (1 - \l') * (1 - \l)$ and $x,y,z \in X$ arbitrary.
Suppose $d_\l(x,y) < \gamma$ and $d_{\l'}(y,z) < \gamma'.$ 
It follows that we can choose $\rho < \l$ and $\rho' < \l'$ such that 
$$
e_\rho (x,y) < \gamma, \ \  e_{\rho'} (y,z) < \gamma'.
$$
Next we can determine $\xi < \ve $ such that 
$$
(1- \ve) < (1 - \xi) < (1 - \l') * (1 - \l) \leq (1 - \rho') * (1 - \rho).
$$
In view of the first step we have $e_\xi (x,z) \leq e_\rho(x,y) +e_{\rho'} (y,z) < \gamma + \gamma'.$
Finally we have
$$
d_\ve (x,z) = \inf_{\ve' < \ve} e_{\ve'} (x,z) \leq e_\xi (x,z) < \gamma + \gamma'.
$$

Next we check {\bf(UH)} for $(X, d_\l)_\l).$ Let $x \not = y,$ then there exists $i \in I$ with $f_i(x) \not = f_i (y)$ and
$\rho$ such that $d^i_\rho (f_i(x) , f_i(y) ) \not = 0.$ It follows that $e_\rho (x,y) \not = 0.$
Remark that from the definition of $e_\rho$ in \eqref{sup} and  {\bf(UD)} for $(X_i, (d^i_\l)_\l)$ it follows that $\mu < \rho$ implies $e_\rho(x,y) < e_\mu (x,y).$
So we have
$$
0\not = e_\rho (x,y) \leq \inf_{\mu < \rho }e_\mu(x,y) = d_\rho (x,y).
$$
Finally we prove that the source 
$$
(f_i: (X, (d_\l )_\l)\rw (X_i, (d^i_\l)_\l))_{i \in I}
$$
is initial. Clearly for $x,y \in X$ in view of {\bf(UD)}  for every $k \in I$  and $\l$ we have
$$
d^k_\l(f_k(x) , f_k (y)) = \inf_{\xi < \l} d^k_\xi(f_k(x) , f_k (y)) \leq  \inf_{\xi < \l} \sup_{i \in I}d^i_\xi(f_k(x) , f_k (y))  = d_\l (x,y), 
$$
and hence all maps $f_i$ are non-expansive.

Next let $h: (Z, (g_\l)_\l) \rw (X, (d_\l )_\l)$ be a function and assume that $f_i \circ h$ is non-expansive on a probabilistic metric space $(Z, (g_\l)_\l).$ 
This means that for all $i \in I$ and for all $\xi$
$$ d^i_\xi (f_i(h(z)), f_i(h(z')) \leq g_\xi(z,z').$$
It follows that with the notation of \eqref{sup} $e_\xi(h(z), h(z')) \leq g_\xi(z,z')$ for all $\xi.$
Applying \eqref{inf} and {\bf(UD)} for $(Z, (g_\l)_\l)$   for $\l$ arbitrary we obtain
$$
d_\l (h(z), h(z')) = \inf_{\xi < \l} e_\xi(h(z), h(z')) \leq  \inf_{\xi < \l} g_\xi(z,z')= g_\l(z,z').
$$
\end{proof}
\end{theorem}

In a similar way as in \ref{Z} and \ref{mono} for the category $p\PM$ of pseudo probabilistic metric spaces with non-expansive maps, obtained by dropping {\bf(UH)} or equivalently (P4),
we obtain, 
\begin{proposition}\label{topolog}
The category $p\PM$ can be isomorphically described as a category with objects $(X, \{d_\l \ | \ 0< \l \leq 1\}),$ sets endowed with a collection of distances numbered by $ 0< \l \leq 1,$ where $\{d_\l \ | \ 0< \l \leq 1\}$
satisfies the following conditions: {\bf(US)}, {\bf(UD)}, {\bf(UT)}  for some continuous t-norm $*$. 

$p\PM$ is topological over $\mathsf{Set} $ and is well-fibred  \cite{AHS}. It is an (Extremal epi, Mono) and cowell-powered category.
\end{proposition}

\begin{proposition}
The category $\PM$ is extremally epireflective in $p\PM.$
\begin{proof}
In view of the previous proposition it is sufficient to show that $\PM$ is closed under mono subobjects and products in $p\PM.$

First let $f: (X,(d_\l)_\l) \rw (Z, (e_\l)_\l)$ be an injective non-expansive map, with $(X,(d_\l)_\l)$ a pseudo probabilistic metric and $(Z, (e_\l)_\l$ a probabilistic metric space. We check {\bf(UH)} for $(X,(d_\l)_\l).$
Let $x \not = y$ in $X$, then we have $f(x) \not = f(y).$ There exists a $\l$ such that $e_\l (f(x), f(y)) > 0.$ Then we have
$$
0< e_\l (f(x), f(y)) \leq d_\l (x,y).
$$
Next we consider a product in $p\PM$
$$(pr_i: (\prod{X_i}, (d_\l )_\l)\rw (X_i, (d^i_\l)_\l))_{i \in I}$$
 with $(X_i, (d^i_\l)_\l))$ a probabilistic metric space for every $i \in I$. Then this source is pointseparating. The explicit formula of $d_\l$ is similar to \eqref{inf} with $f_i = pr_i$ for $i \in I.$ The proof of {\bf(UH)} for $(\prod{X_i}, (d_\l)_\l))$ is similar to the one for the case of 
{\bf(UH)} in \ref{mono}.
\end{proof}
\end{proposition}

Since $\PM$ is monotopological we have the following results \cite{Nel}, \cite{AHS}.
\begin{proposition}
For any non-expansive map $f$ in $\PM$ the following are equivalent:
\begin{enumerate}
\item $f$ is a quotient
\item $f$ is a regular epi
\item $f$ is an extremal epi.
\end{enumerate}
The category $\PM$ is an $(extremal \ epi, mono)$-category.
Since in $\PM$ the monomorphisms are the injective non-expansive maps, $\PM$ is well-powered.
\end{proposition}
The characterization of the epimorphisms and extremal monomorphisms and the study of cowell-poweredness will be treated in the next section.

\section{the regular closure related to $\PM$}
Given $(X, (d_\l)_\l),$ \emph{the strong uniformity} \cite {SS} $\mathcal{U}_0$ on $X$ is generated by $$\{ U_\l^\gamma | 0 < \gamma, 0< \l \leq 1 \},$$  where $$U_\l^\gamma = \{ (x,y) | d_\l (x,y) < \gamma\}.$$ In terms of the $\mathsf{UG}$  space on $X$ \cite{CL2}, generated by $(X, (d_\l)_\l),$ the uniform space $(X,\uuu_0)$ is the uniform coreflection. The \emph{strong topology} $\mathcal{T}_0$ on $X$ is the underlying topology of $\mathcal{U}_0.$ Clearly a non-expansive map between probabilistic metric spaces is uniformly continuous for the associated strong uniformities and continuous for the associated strong topologies.

In this section we show that the objects in $\PM$ coincide with the $T_0$-objects in $p\PM.$
Further we prove that they are also characterized by the $T_0$-property in the strong topology.

As $\PM$ is an (extremal) epireflective subcategory of the topological category $p\PM,$ it induces a regular closure operator \cite{DG}, \cite{DT}, \cite{DGT}. We prove that the regular closure is equal to the closure in the strong topology and we deduce characterizations of the epimorphisms and the regular monomorphisms. We show that $\PM$ is a cowell-powered $(\mathcal{E}, \mathcal{M})$-category for $\mathcal{E}$ the epis and $\mathcal{M}$ the regular monos.\\

We refer to the usual notion of a $T_0$-object in a topological category \cite{M} and apply it to $p\PM.$ 
$(X, (d_\l)_\l)$ is a \emph{$T_0$-object} if and only if every non-expansive map from $I_2,$ the indiscrete pseudo probabilistic metric space with two points (with only one distance $d=0$), to $(X,( d_\l)_\l)$ is constant.

\begin{proposition}
The following are equivalent for $(X, (d_\l)_\l)$ in $p\PM$
\begin {enumerate}
\item $(X, (d_\l)_\l)$ is a $T_0$-object
\item $(X, (d_\l)_\l)$ satisfies {\bf(UH)}
\item The induced strong uniformity is $T_0$
\item The induced strong topology is $T_0$
\end{enumerate}
\begin{proof}
(1) $\Rightarrow$ (2) Assume $x \not = y$ and define $f: I_2 \rw (X, d_\l)_\l)$ defined by $f(0)= x$ and $f(1) = y.$ By (1) this function is not a morphism. Hence there exists $\l$ with $d_\l (x,y) \not = 0.$\\
(2) $\Rightarrow$ (3) Let $x \not = y$ and $\l$ with $d_\l (x,y) \not = 0.$ With $\gamma$ satisfying $0 < \gamma < d_\l(x,y)$ we clearly have $(x,y) \not \in U_\l^\gamma$.\\
(3) $\Leftrightarrow$ (4) is evident.\\
(3) $\Rightarrow$ (1) Suppose $f: I_2 \rw (X, (d_\l)_\l)$ is non-expansive but not constant, $f(0) \not = f(1)$ in $X$. By non-expansiveness of $f,$ for every $\l$ we have $d_\l(f(0), f(1)) \leq d(0,1) = 0$. It follows that $(f(0), f(1)) \in U_\l^\gamma$ for all $\l$ and $\gamma.$
\end{proof}
\end{proposition}

Next we consider the following closure operator on $\PM$.
For a probabilistic metric space $(X, (d_\l)_\l)$ consider the closure $\cl: 2^X \rw 2^X$ mapping $A \subseteq X$ to the closure in the strong topology. As it is a topological closure operator and non-expansive maps are continuous, $\cl$ satisfies all properties listed in \cite{DT}, extension, monotonicity, continuity, it is idempotent, hereditary, grounded and additive.

We first give some equivalent descriptions. Let $(Y, (d_\l)_\l)$ a probabilistic metric space, $A \subseteq Y$ and $y \in Y$ and $0< \l \leq 1$, then we denote
$$
d_\l(y, A) = \inf_{a\in A} d_\l(y,a).
$$

\begin{proposition}\label{equivtop}
Let $(Y, (d_\l)_\l)$ a probabilistic metric space, $A \subseteq Y$ and $y \in Y$.
The following are equivalent:
\begin{enumerate}
\item $y \in \cl(A)$
\item $\forall\  0< \l \leq 1, \forall  \ 0 < \gamma < \infty, \ d_\l(y,A) < \gamma$
\item $\forall\  0< \l \leq 1, d_\l(y,A) =0$
\item $\forall \ 0< \rho \leq 1, \  \ d_\rho(y,A) < \rho.$
\end{enumerate}
\begin{proof}
(1) $\Rightarrow$ (2) Let $0< \l \leq 1$ and $0< \gamma < \infty,$ consider the neighborhood of $y$ in the strong topology 
$V_\l^\gamma = \{ x \ | \ d_\l(x,y) < \gamma \}$ and choose $a \in A \cap V_\l^\gamma.$ Then clearly $d_\l(y,a) < \gamma.$\\
(2) $\Leftrightarrow$ (3) is clear.\\
(2) $\Rightarrow$ (4) Let $ 0< \rho \leq 1$. With $\l = \gamma = \rho$ the result follows.\\
(4) $\Rightarrow$ (1) Let  $0< \l \leq 1,  \ 0 < \gamma < \infty.$ Choose $0 < \rho < \l \wedge \gamma.$ By (4) and {\bf(UD)} we have $d_\l (y ,A) \leq d_\rho(y,A) < \rho <\gamma.$ So there exists $a \in A$ satisfying $a \in V_\l^\gamma.$

\end{proof}
\end{proposition}

Recall that for $(Y, (d_\l)_\l)$ a probabilistic metric space, $A \subseteq X$ the \emph{regular closure} of $A$ is defined by
$$
\text{reg}(A) = \{ y \in Y \ | \ \forall  u,v : Y \rw Z, \  u_{|_A} = v_{|_A} \Rightarrow u(y) = v(y) \}
$$
with $u,v$ non-expansive and $ Z \in \PM.$

Next we prove that the two closures coincide.

\begin{theorem}\label{iso}
For $(Y, (d_\l)_\l)$ a probabilistic metric space, $A \subseteq X$ we have 
$$
\cl(A) = \text{reg}(A).
$$

\begin{proof}
First let $y \in \cl(A)$ and assume $u,v : (Y, (d_\l)_\l) \rw (Z, (e_\l)_\l)$ are non-expansive with $(Z, (e_\l)_\l) \in \PM$ with $u_{|_A} = v_{|_A}.$ Suppose $u(y) \not = v(y)$ then in view of {\bf(UH)} in $(Z, (e_\l)_\l),$ $\exists \  0 < \ve \leq 1$ such that $e_\ve(u(y) , v(y)) \not = 0.$
Choose $0 < \l \leq 1$ satisfying $$2 \l < e_\ve (u(y) , v(y)) \  \text{and } \ (1-\ve) < (1-\l) * (1-\l).$$
Applying (4) from \ref{equivtop} there exists $a \in A$ satisfying $d_\l(y,a) < \l$. Applying {\bf(UT)} we get a contradiction with the choice of $\l$.
 \begin{eqnarray*}
 e_\ve(u(y),v(y)) &\leq & e_\l(u(y), u(a)) + e_\l(u(a), v(y))\\
 &=&e_\l(u(y), u(a)) + e_\l(v(a), v(y))\\
 &\leq& d_\l(y,a) + d_\l(a,y)\\
 &=&2d_\l(y,a)\\
 &<& 2\l.\\
 \end{eqnarray*}
 
 Next we show the other inclusion. Assume $Y \setminus \cl{A} \not = \emptyset.$ 
 We define
 $$
 Z = (Y \setminus \cl (A)) \cup \{0\} 
 $$
 and $(e_\l)_\l $ on $Z$ defined as follows.
For $0 < \l \leq 1,$ $y,z \in Y \setminus \cl(A)$
 $$
  \begin{cases}  e_\l(y, z) = d_\l(y,z) \\
 e_\l(0,y) =  e_\l(y, 0) =d_\l(y,A)&\\ 
 e_\l(0,0) = 0. \end{cases}
$$
We check that $(Z, (e_\l)_\l)$ is a probabilistic metric space. Clearly for every $\l$ the distance $e_\l$ satisfies {\bf(US)}.
Also {\bf(UD)} is clear. The first case is $(y, z)$ with $y,z \in Y \setminus \cl(A).$
That $e_\l (y, z) = \inf_{\mu < \l} e_\mu(y, z) $ follows at once from the {\bf(UD)} property of $(Y, (d_\l)_\l).$ The second case is $(y, 0)$ with $y \in Y \setminus \cl(A).$ In this case we have
$$e_\l(y, 0) = d_\l(y,A)= \inf_{a \in A}d_\l(y,a) = \inf_{a \in A} \inf_{\mu < \l} d_\mu (y,a) = \inf_{\mu < \l} d_\mu(y,A)=  \inf_{\mu < \l} e_\mu(y, 0).$$
Next we check {\bf(UT)}. Consider $\ve, \l, \l'$ in $]0,1]$ satisfying $(1-\ve) < (1-\l') * (1-\l).$
For $y,z,x \in Y \setminus \cl(A)$ we have 
$$
e_\ve(y,z) = d_\ve(y,z) \leq d_\l(y,x) +d_{\l'}(x,z) = e_\l(y,x) +e_{\l'}(x,z).
$$
For $y,z \in Y \setminus \cl(A)$ and $0$ we have
$$
e_\ve(y,z) = d_\ve(y,z) \leq d_\l(y,a) +d_{\l'}(a,z) 
$$
whenever $a \in A.$ Hence 
$$
e_\ve(y,z)  \leq d_\l(y,A) +d_{\l'}(z,A) = e_\l(y,0) +e_{\l'}(0,z).
$$
Finally for $y\in Y \setminus \cl(A),$ $0$ and $z\in Y \setminus \cl(A),$  we have $d_\ve(y,a) \leq d_\l(y,z) + d_{\l'} (z,a)$ for all $a \in A.$ This implies
$$
e_\ve(y,0) = d_\ve(y,A) \leq e_\l(y,z) + d_{\l'} (z,A) = e_\l(y,z) + e_{\l'} (z,0).
$$
To check the property {\bf(UH)} let $y \not =z \in Y \setminus \cl(A)$. Take $\l$ such that $d_\l(y,z) \not = 0$ then clearly $e_\l(y,z) \not = 0.$
Let $y \in Y \setminus \cl(A).$ By (3) in \ref{equivtop} we have $d_\l(y,A) \not = 0.$ This implies $e_\l(y,0) \not = 0.$

Consider the following pair of maps $u,v: (Y, (d_\l)_\l) \rw (Z, (e_\l)_\l),$ with $v$ identically zero and $u$ defined by
$$
  \begin{cases}  u(y) = y \ \  \text{on} \ \ Y \setminus \cl(A) \\
 u(a) = 0  \  \ \text{on} \  \  \cl(A).  \end{cases}
$$
We show that $u$ is non-expansive (it is even a quotient in $\PM$). Let $\l $ be arbitrary. For $z,y \in Y \setminus \cl(A)$ we have $$e_\l(u(z), u(y)) = e_\l(z,y) = d_\l(y,z).$$\\
For $z \in Y \setminus \cl(A), a \in \cl(A)$ we have $$e_\l(u(z), u(a)) = e_\l(z,0) = d_\l(z, A) \leq d_\l(z,a).$$
Finally for $a  , a'  \in \cl(A)$ we have $$e_\l(u(a), u(a')) = 0 \leq d_\l(a,a').$$ 
So we can conclude that $u$ and $v$ coincide on $\cl{A}$ but  $u(y) \not = v(y)$ for $y \in Y \setminus \cl(A).$ 

\end{proof}
\end{theorem}

We immediately have the following corollaries \cite{DGT}, \cite{DT}, \cite{CGT}.

\begin{corollary}\label{epi}
A non-expansive map in $\PM$  $f: (X,(d_\l)_\l) \rw (Y,(e_\l)_\l)$ is an epimorphism in $\PM$ if and only if it is a dense map, meaning $\cl (f(X)) = Y.$
\end{corollary}

Since the closure operator $\cl$ on $\PM$ is (weakly) hereditary it also enables us to describe the regular mono's.

\begin{corollary}
The following classes of non-expansive maps in $\PM$ coincide:
\begin{enumerate}
\item The class of all regular monomorphisms
\item The class of all extremal monomorphisms
\item The class of all $\cl$-closed embeddings
\end{enumerate}
\end{corollary}

Since the closure operator $\cl$ is idempotent and (weakly) hereditary we have the following.
\begin{corollary}
$\PM$ is an $(\eee^{\cl}, \mmm^{\cl})$ category, with $\eee^{\cl}$ the dense maps and $ \mmm^{\cl}$ the closed embeddings
\end{corollary}

\begin{proposition}
$\PM$ is cowell-powered.
\begin{proof}
Since the closure on $(Y, (e_\l )_\l) \in \PM$ is the same as the one in the strong topological space $(Y, \ttt_0)$ which is a Hausdorff space, the method for proving that $\mathsf{Haus}$ is cowell-powered \cite{P} is also relevant here. \\

Given $f: (X, (d_\l )_\l)) \rw (Y, (e_\l )_\l)$ non-expansive and dense, then for every $y \in Y$ the neighborhoodfilter $\vvv(y)$ of $y$ in the strong topology has a trace on $f(X).$ 
Let $i: f(X) \rw Y$ be the inclusion map, because of the Hausdorff property of $(Y, \ttt_0)$ different points in $Y$ create different traces $i^{-1} (\vvv(y)$.  
The map
$$
g: Y \rw \ppp(\ppp(f(X))),
$$
 defined by $g(y) = i^{-1} (\vvv(y)$ is therefore injective. Since there exists an injective map from $f(X)$ to $X$, there is an injective map
$$
h: \ppp(\ppp(f(X))) \rw \ppp(\ppp(X)).
$$
Hence the map 
$$
h \circ g: Y \rw \ppp(\ppp(X))
$$
is injective. Let $Z \subseteq \ppp(\ppp(X))$ be its image. Then the structure $(e_\l)_\l$ can be transported to $Z$ such that $h \circ g: Y  \rw Z$
becomes an isomorphism in $\PM.$ Hence there exists a representative set of epimorphisms.
\end{proof}

\end{proposition}

\section{The embedding of $\mathsf{Met}^\infty$ in $\PM$} 

In this section we consider the category $\met^\infty$ of extended metric spaces with non-expansive maps.
There is an embedding of $\met^\infty$ in $\PM$ sending an extended metric space $(X,d)$ to the probabilistic metric space $(X, (d_\l)_\l)$ with $d_\l = d$ for all $\l.$ Clearly this is a full embedding. We show that it is both concretely coreflective and reflective in $\PM$.

\begin{theorem}
$\met^\infty$ is concretely coreflectively embedded in $\PM$.
\begin{proof}
Let $(X, (d_\l)_\l)$ be an arbitrary probabilistic metric space. Let $(X,d)$ be defined as 
$$
d(x,y) = \sup_\l d_\l(x,y).
$$
We clearly have $d(x,x) = 0, d(x,y) = d(y,x)$ and for $x \not = y$ we have $d(x,y) >0.$
That the triangle inequality holds follows from {\bf(UT)}.
For $\ve$ arbitrary determine $\l, \l'$ satisfying $(1-\ve) < (1-\l) * (1-\l').$ For $x,y,z \in X,$
$$
d_\ve(x,y) \leq d_\l(x,z) + d_{\l'} (z,y) \leq d(x,z) + d(z,y).
$$
Hence $d(x,y) \leq d(x,z) + d(z,y)$.

Let $(Z,e)$ be an extended metric space embedded in $\PM$ and $f: (Z,e) \rw (X, (d_\l)_\l)$ be a non-expansive map. Then the map $\overline{f} : (Z,e) \rw (X, d)$ is non-expansive since $d_\l(f(z), f(z')) \leq e(z,z')$ for all $\l,$ making the following diagram commute.

\begin{align*}
\xymatrix{(Z,(e_\l)_\l)\ar[r]^{\overline{f}}\ar[dr]_{f} & (X,d) \ar[d]^{1_X} \\ & (X,(d_\l)_\l)}
\end{align*}
\end{proof} 
\end{theorem}

\begin{theorem}
$\met^\infty$ is reflectively embedded in $\PM$.
\begin{proof}
Let $(X, (d_\l)_\l)$ be an arbitrary probabilistic metric space. Consider $(X,d_1)$ with $d_1$ the distance for $\l = 1.$ Applying {\bf(US)} we have that the distance satisfies $d_1(x,y) = d_1(y,x).$ In order to force the triangle inequality we use the standard method and put
$$
\tilde{d_1} (x,y) = \inf \{ \sum_{j=1}^{n} d_1(x_i, x_i+1) \ | \ n \in \N_0, x_1 = x, x_n = y \}.
$$
Then the modification $(X,\tilde{d_1})$ is a reflection, being the largest extended pseudo metric space smaller than $(X,d_1)$.
We now have the following commutative diagram given $(Z,e)$ an extended metric space embedded in $\PM$ and $f: (X, (d_\l)_\l) \rw (Z,e)$ a non-expansive map.

$$
\xy
\xymatrix@M+2pt{ (X, (d_\l)_\l)  \ar[r]^{f} \ar[d]_{ 1_X} &  (Z,e)  \\  (X,d_1) \ar[r]_{1_{X}} \ar[ur]_{\overline{f}}
& (X,\tilde{d_1}) \ar[u]_{\tilde{f}}}
\endxy
$$

That $\overline{f}$ and $\tilde{f}$ are indeed non-expansive follows easily:
$$e(f(x), f(y)) \leq d_\l(x,y),  \forall \l \Rightarrow e(f(x), f(y)) \leq d_1(x,y) \Rightarrow e(f(x), f(y)) \leq \tilde{d_1}(x,y).$$
$(X,\tilde{d_1})$ is a pseudo metric space but it need not satisfy separation. Hence we apply the standard method to force separation by defining a suitable quotient.

$$x \simeq x' \Leftrightarrow \tilde{d_1}(x,x')  = 0$$ is an equivalence relation and
let $\varphi: X \rw X^*$ be the associated quotient map mapping $x$ to its equivalence class $x^*$.
Then  on $X^*$ we put 
$$
\tilde{d_1}^*(x^*, y^*) = \tilde{d_1} (x,y).
$$
This definition is unambiguous in view of the triangle inequality of $\tilde{d_1},$ satisfies symmetry, the triangle inequality and separation.
Next we consider the following diagram with $(Z,e)$ an extended metric space and $\tilde{f} : (X, \tilde{d_1}) \rw (Z,e)$ non-expansive and we show that a unique function $f^*$ exists, making the diagram commutative.
$$
\xymatrix{(X, \tilde{d_1})  \ar[d]_{\varphi} \ar[r]^{\tilde{f}}&(Z,e)  \\
(X,\tilde{d_1}^*) \ar[ru]_{f^*}&
}
$$ 
That $f^*$ exists follows from 
$$
\varphi(x)= \varphi(x') \Rightarrow x \simeq x' \Rightarrow \tilde{d_1}^* (x^*,x'^*) = \tilde{d_1}(x,x') = 0
$$ 
and since $e( \tilde{f}(x), \tilde{f}(x')) \leq \tilde{d_1} (x,x') = 0$ we can conclude that $\tilde{f} (x) = \tilde{f} (x').$
Uniqueness of $f^*$ follows from the surjectivity of $\varphi.$ The map $\varphi$ is a quotient in $p\met^\infty$ and hence $f^*$ is non-expansive.
\end {proof}
\end{theorem}

\vspace{1cm}

\vspace{1cm}
\noindent {E. Colebunders \\
Department of Mathematics and Data Science, Vrije Universiteit Brussel, Pleinlaan 2, 1050 Brussel, Belgi\"{e}\\  \emph{evacoleb@vub.be}\\
Department of Mathematics, Universiteit Antwerpen, Middelheimlaan 1, 2020 Antwerpen, Belgi\"{e}}\\
\emph{eva.colebunders@uantwerpen.be}\\
 \\
R. Lowen \\
Department of Mathematics, Universiteit Antwerpen, Middelheimlaan 1, 2020 Antwerpen, Belgi\"{e}}\\
\emph{bob.lowen@uantwerpen.be}

\end{document}